\documentclass[a4paper,12pt]{amsart}

\usepackage{epsfig}
\usepackage{amsthm,amsfonts}
\usepackage{amssymb,graphicx,color}
\usepackage[all]{xy}
\usepackage{verbatim}
\usepackage{hyperref}
\usepackage{enumitem}

\usepackage[a4paper,top=3cm,bottom=3cm,left=3.5cm,right=3.5cm]{geometry}

\newtheorem{theorem}{Theorem}[section]
\newtheorem*{theorem*}{Theorem}

\newtheorem{proposition}[theorem]{Proposition}
\newtheorem{definition}[theorem]{Definition}

\newtheorem{remark}[theorem]{Remark}

\newcommand{\R}{\mathbb{R}}

\newcommand{\C}{\mathbb{C}}
\newcommand{\CC}{\mathbb{C}}

\newcommand{\bbP}{\mathbb{P}}
\newcommand{\Z}{\mathbb{Z}}

\newcommand{\N}{\mathbb{N}}

\newcommand{\T}{\widetilde}

\begin{document}

\title[Multiplicity and degree as bi-Lipschitz invariants]
{Multiplicity and degree as bi-Lipschitz invariants for complex sets}

\author[J. Fern\'andez de Bobadilla]{Javier Fern\'andez de Bobadilla}
\address{Javier Fern\'andez de Bobadilla:  
(1) IKERBASQUE, Basque Foundation for Science, Maria Diaz de Haro 3, 48013, 
    Bilbao, Bizkaia, Spain
(2) BCAM  Basque Center for Applied Mathematics, Mazarredo 14, E48009 Bilbao, 
Basque Country, Spain } 
\email{jbobadilla@bcamath.org}

\author[A. Fernandes]{Alexandre Fernandes}
\author[J. E. Sampaio]{J. Edson Sampaio}
\address{Alexandre ~Fernandes and J. Edson ~Sampaio - 
Departamento de Matem\'atica, Universidade Federal do Cear\'a, 
Rua Campus do Pici, s/n, Bloco 914, Pici, 60440-900, 
Fortaleza-CE, Brazil}  

 \email{alex@mat.ufc.br}
 \email{edsonsampaio@mat.ufc.br}

\thanks{The first named author is partially supported by IAS and by ERCEA 615655 NMST Consolidator Grant, MINECO by the project reference MTM2013-45710-C2-2-P, by the Basque Government through the 
BERC 2014-2017 program, by Spanish Ministry of Economy and Competitiveness MINECO: BCAM Severo Ochoa excellence accreditation SEV-2013-0323, by
Bolsa Pesquisador Visitante Especial (PVE) - Ciencias sem Fronteiras/CNPq Project number:  401947/2013-0 and by Spanish MICINN project MTM2013-45710-C2-2-P. 
The second named author was partially supported by CNPq-Brazil grant 302764/2014-7}

\begin{abstract}
We study invariance of multiplicity of complex analytic germs and degree of complex affine sets under outer bi-Lipschitz 
transformations (outer bi-Lipschitz homeomorphims of germs in the first case and outer bi-Lipschitz homeomorphims at infinity 
in the second case). We prove that invariance of multiplicity in the local case is equivalent to invariance of degree in the global case. 
We prove invariance for curves and surfaces. In the way we prove invariance of the tangent cone and relative multiplicities at infinity 
under outer bi-Lipschitz homeomorphims at infinity, and that the abstract topology of 
a homogeneous surface germ determines its multiplicity.
\end{abstract}

\maketitle

\section{Introduction}

We study invariance of multiplicity of complex analytic germs and degree of complex affine sets under outer bi-Lipschitz 
transformations: outer bi-Lipschitz homeomorphisms of germs in the first case and outer bi-Lipschitz homeomorphisms at infinity 
in the second case (see Definition~\ref{def_inf}). 

The local problem may be seen as a bi-Lipschitz version of Zariski multiplicity problem~\cite{Zariski:1971}, that asks whether 
the embedded topology of complex hypersurface germs determine their multiplicity. It is well known that the abstract topological
type of hypersurface germs (i.e the topology of the intersection of the germs with a smal ball) does not determine the 
multiplicity. This is clear for curves, since their abstract topology is too simple, but also in the case of surfaces, 
where the abstract topology is very rich, examples are known (see Example 2.22 at~\cite{Nemethi:1999}).

In contrast with this situation we conjecture, as we will make precise shortly, that the local and at infinity outer bi-Lipschitz 
geometry (which does not take into account the embedding) determine the multiplicity and degree of complex sets respectively. 
It is known that the inner bi-Lipschitz geometry does not determine multiplicity (for example complex curves).

\begin{enumerate}[leftmargin=0pt]
\item[]{\bf Conjecture \~A1($d$)} Let $X\subset \C^n$ and $Y\subset \C^m$ be two complex analytic sets with $\dim X=\dim Y=d$. 
If their germs at $0\in\C^n$ and $0\in\C^m$, respectively, are outer bi-Lipschitz homeomorphic, i.e. there exists an outer bi-Lipschitz 
homeomorphism $\varphi\colon (X,0)\to (Y,0)$, then  their multiplicities $m(X,0)$ and $m(Y,0)$ are equal.
\item[]{\bf Conjecture  A1($d$)} Let $X\subset \C^n$ and $Y\subset \C^m$ be two complex algebraic sets with $\dim X=\dim Y=d$. 
If $X$ and $Y$ are outer bi-Lipschitz homeomorphic at infinity (see Definition~\ref{def_inf}), then we have the equality ${\rm deg}(X)={\rm deg}(Y)$.
\end{enumerate}

Conjecture \~A1($d$) was approached by some authors.
G. Comte~\cite{Comte:1998} proved that the multiplicity of complex analytic germs  is invariant under outer bi-Lipschitz homeomorphism
with Lipschitz constants close enough to 1 (this is a severe assumption). Neumann and Pichon in \cite{N-P}, with previous contributions of 
Pham and Teissier in \cite{P-T} and Fernandes in \cite{F}, proved that the outer bi-Lipschitz geometry of plane curves determine the Puiseux pairs, and 
as a consequence proved \~A1($1$). More recently, W. Neumann and 
A. Pichon in \cite{NeumannP:2016} showed that the multiplicity is an outer bi-Lipschitz invariant in the case of normal surface 
singularities, as a consequence of a very detailed and involved study of the outer bi-Lipschitz geometry for that class. In the 
non-normal surface case the only partial contribution is the fact that the {\em embedded} bi-Lipschitz geometry determines 
multiplicity in the {\em hypersurface} case (see \cite{FernandesS:2016}).
Conjecture A1($d$) is largely unexplored up to our knowledge.

Our main results are the following. We prove that Conjecture \~A1($d$) is equivalent to Conjecture A1($d$) (Theorem \ref{main-theorem}). 
We prove the conjectures for curves and surfaces ($d=1,2$) (Theorems~\ref{curves} and~\ref{main-result}. 
For general $d$ we prove in Theorem \ref{application1} 
that Conjecture A1 holds for algebraic 
hypersurfaces in $\C^n$ whose all irreducible components of their tangent cones at infinity have singular locus with 
dimension $\leq 1$ and, as an immediate corollary of it, we obtain that degree of complex algebraic surfaces in $\C^3$ is an embedded bi-Lipschitz invariant at infinity.

The way to reach this results is to prove that the outer Lipschitz geometry at
infinity determines the tangent cone at infinity and the relative multiplicities at infinity (Theorem \ref{multiplicities}). 
These are versions ``at infinity''
of the corresponding results for germs in \cite{Sampaio:2016} and \cite{FernandesS:2016} respectively. 
In order to have an idea of the ingredients of the statement
let us remark that in the hypersurface case the tangent cone at infinity is the set defined by the 
highest degree form of the defining equation. The relative multiplicities at infinity are the exponents appearing in the 
factorization in irreducible components of the highest degree form. Precise definitions are in Section~\ref{section:preliminaries}.

The case $d=1$ comes very easily from the last mentioned result. For the $d=2$ case the new idea is to find 
the degree of a homogeneous irreducible affine algebraic set $S$ as the torsion part of a homology group of $S\setminus\{O\}$.
We use the Leray spectral sequence associated with its projectivization for that purpose. In particular we show that the abstract topology of 
a homogeneous surface germ determines its multiplicity.

The organization of the paper is as follows. In Section \ref{section:preliminaries} we recall the necessary basic definitions and introduce 
relative multiplicities at infinity. In Section~\ref{section:mainresults} we prove the results described above.

\bigskip


\section{Preliminaries}\label{section:preliminaries}

\subsection{Multiplicity, degree and tangent cones}

See \cite{Chirka:1989} for a definition of multiplicity for higher codimension analytic germs 
in $\C^n$.

\begin{definition}
 Let $A$ be a closed algebraic subset in $\C^n$. We define {\bf the degree of} $A$ to be the degree of its projective completion~\cite{Chirka:1989}.
\end{definition}

\begin{remark}
If $A$ is a homogeneous algebraic set in $\C^n$ (i.e defined by homogeneous polynomials), its degree coincides with its multiplicity at the 
origin of $\C^n$.
\end{remark}   

Now we set the exact notion of tangent cone that we will use along the paper and we list some of its properties.

\begin{definition}
Let $A\subset \R^n$ be an unbounded subset. We say that $v\in \R^n$ is a {\bf tangent vector of} $A$ {\bf at infinity} if there is a sequence of points $\{x_i\}_{i\in \N}\subset A$ such that $\lim\limits_{i\to \infty} \|x_i\|=+\infty $ and there is a sequence of positive numbers $\{t_i\}_{i\in \N}\subset\R^+$ such that 
$$\lim\limits_{i\to \infty} \frac{1}{t_i}x_i= v.$$
Let $C_{\infty }(A)$ denote the set of all tangent vectors of $A$ at infinity. This subset $C_{\infty }(A)\subset\R^n$ is called 
{\bf the tangent cone of} $A$ {\bf at infinity}.
\end{definition}

\begin{proposition}[Proposition 4.4 in \cite{FernandesS:2017}]
Let $Z\subset \R^n$ be an unbounded semialgebraic set. A vector $v\in\R^n$ belongs to $C_{\infty }(Z)$ if, and only if, there exists a continuous semialgebraic  curve $\gamma\colon (\varepsilon ,+\infty )\to Z$  such that $\lim\limits _{t\to +\infty }|\gamma(t)|=+\infty $ and $\gamma(t)=tv+o_{\infty }(t),$ where $g(t)=o_{\infty }(t)$ means $\lim\limits _{t\to +\infty }\frac{g(t)}{t}=0$.
\end{proposition}

Let $X\subset\C^n$ be a complex algebraic subset. Let $\mathcal{I}(X)$ be the ideal of $\C[x_1,\cdots,x_n]$ given by the polynomials which
vanishes on $X$. For each $f\in\C[x_1,\cdots,x_n]$, let us denote by $f^*$ the maximum degree form of $f$. Define $\mathcal{I}^*(X)$ to be 
generated by the $f^*$ when $f\in \mathcal{I}(X)$.

\begin{proposition}[Theorem 1.1 in \cite{LeP:2016}]\label{algebricity}
Let $X\subset\C^n$ be a complex algebraic subset. Then, $C_{\infty }(X)$ is the affine algebraic subset defined by $\mathcal{I}^*(X)$.
We emphasize that we take $C_{\infty }(X)$ as a set, with reduced structure.
\end{proposition}

Among other things, this result above says that  tangent cones at infinity of complex algebraic sets in $\C^n$ are complex algebraic subsets as well.

\subsection{Outer bi-Lipschitz homeomorphism at infinity}
All the Euclidean subsets are equipped with the induced Euclidean distance (outer metric). So, all the Lipschitz mappings mentioned here are supposed to be Lipschitz with respect to the outer metric, this is why they are called outer Lipschitz or bi-Lipschitz mappings 
\begin{definition}
\label{def_inf}
Let $X\subset \R^n$ and $Y\subset\R^m$ be two subsets. We say that $X$ and $Y$ are {\bf outer bi-Lipschitz homeomorphic at infinity}, if there exist compact subsets $K\subset\R^n$ and $\widetilde K\subset \R^m$ and an outer bi-Lipschitz homeomorphism $\phi \colon X\setminus K\rightarrow Y\setminus \widetilde K$.
\end{definition}

We finish this subsection reminding the invariance of the tangent cone at infinity under outer bi-Lipschitz homeomorphisms at infinity.

\begin{proposition}[Theorem 4.5 in \cite{FernandesS:2017}]\label{inv_cones}
Let $X\subset\R^n$ and $Y\subset\R^m$ be unbounded semialgebraic subsets. If $X$ and $Y$ are outer bi-Lipschitz homeomorphic at infinity, then there is an outer bi-Lipschitz homeomorphism $d\varphi\colon C_{\infty }(X)\to C_{\infty }(Y)$ with $d\varphi(0)=0$.
\end{proposition}

\subsection{Relative multiplicities at infinity}\label{section:lelong}

Let $X\subset \C^n$ be a complex algebraic set with $p=\dim X\geq 1$. 
Let $X_1,\cdots,X_r$ be the irreducible components of $C_{\infty }(X)$. In the primary decomposition of $I^*$ there is a unique primary 
ideal $\mathcal{Q}_i$ which defines $X_i$ as a set. Let $\mathcal{P}_i$ be the minimal prime associated with $\mathcal{Q}_i$. 
The relative multiplicity associated with $X_i$ can be defined algebraically as 
the length of the algebra 
$$(\CC[x_1,...,x_n]/\mathcal{Q}_i)_{(\mathcal{P}_i)}.$$  

Below we give a (equivalent) geometric definition of relative multiplicities which suits better our purposes.

Let $\pi\colon\C^n\to \C^p$ be a linear projection such that 
$$\pi^{-1}(0)\cap(C_{\infty }(X))=\{0\}.$$ 
Therefore, $\pi|_{ X}\colon X\rightarrow \C^p$ (resp. $\pi|_{ C_{\infty }(X)}\colon C_{\infty }(X)\rightarrow \C^p$) is a ramified cover
with degree equal to ${\rm deg} (X)$ (resp. ${\rm deg} (C_{\infty }(X))$) (see \cite{Chirka:1989}, Corollary 1 in the page 126). 
In particular, $\pi|_{X_j}\colon X_j\rightarrow \C^p$ is a ramified cover with degree equal to ${\rm deg} (X_j)$, for each $j=1,\cdots,r$. 
Moreover, if the ramification locus of $\pi|_X$ (resp. $\pi|_{ C_{\infty }(X)}$) is not empty, it is a codimension $1$ complex algebraic 
subset $\sigma(X)$ (resp. $\sigma(C_{\infty }(X))$) of $\C^p$. Let us denote 
$\Sigma=\pi|_X^{-1}(\sigma(X))$ and $\Sigma'=\pi|_{C_{\infty }(X)}^{-1}(\sigma(C_{\infty }(X)))$.

Fix $j\in \{1,\cdots,r\}$. For a point $v\in X_j\setminus (C_{\infty }(\Sigma)\cup C_{\infty }(\Sigma'))$, let $\eta, R >0$ such that 
$$C_{\eta,R }(v'):=\{w\in \C^p|\, \exists t>0; \|tv'-w\|\leq \eta t \}\setminus B_R(0)\subset \C^p\setminus \sigma(X)\cup \sigma(C_{\infty }(X)),$$
where $v'=\pi(v)$. Thus, the number of connected components of $\pi|_X^{-1}(C_{\eta,R }(v'))$ (resp. $\pi|_{ X_j}^{-1}(C_{\eta,R }(v'))$) is equal to ${\rm deg} (X)$ (resp. ${\rm deg} (X_j)$). Moreover, there exist a connected component $V$ of $\pi|_{ X_j}^{-1}(C_{\eta,R }(v'))$ such that $v\in V$ and a compact subset $K\subset \C^n$ such that for each connected component $A_i$ of $\pi|_X^{-1}(C_{\eta,R }(v'))$, we have $C_{\infty }(A_i)\cap (\C^n\setminus K)\subset \pi|_{ C_{\infty }(X)}^{-1}(C_{\eta,R }(v'))$. Then, we denote by $k_X^{\infty }(v)$ to be the number of connected components $A_i'$s such that $C_{\infty }(A_i)\cap (\C^n\setminus K)\subset V$. By definition, we can see that $k_X^{\infty }$ is locally constant and as $X_j\setminus (C_{\infty }(\Sigma)\cup C_{\infty }(\Sigma'))$ is connected, $k_X^{\infty }$ is constant on $X_j\setminus (C_{\infty }(\Sigma)\cup C_{\infty }(\Sigma'))$. Thus, we define $k_X^{\infty }(X_j)=k_X^{\infty }(v)$. In particular, $k_X^{\infty }(w)=k_X^{\infty }(v)$ for all $w\in \pi^{-1}(v')\cap X_j$ and, therefore, we obtain
\begin{equation}\label{kurdyka-raby}
{\rm deg}(X)=\sum\limits_{j=0}^r k_X^{\infty }(X_j)\cdot{\rm deg}(X_j).
\end{equation}
Notice that $k_X^{\infty }$ does not depend of $\pi$.

\section{Main results}\label{section:mainresults}

Let $Z\subset\R^{\ell}$ be a path connected subset. Given two points $q,\tilde{q}\in Z$, we recall that the \emph{inner distance} in $Z$ 
between $q$ and $\tilde{q}$ is the number $d_Z(q,\tilde{q})$ below:
$$d_Z(q,\tilde{q}):=\inf\{ \mbox{length}(\gamma) \ | \ \gamma \ \mbox{is an arc on} \ Z \ \mbox{connecting} \ q \ \mbox{to} \ \tilde{q}\}.$$

\begin{theorem}\label{multiplicities}
Let $X\subset\C^n$ and $Y\subset\C^m$ be complex algebraic subsets, with pure dimension $p=\dim X=\dim Y$, and let $X_1,\dots,X_r$ and $Y_1,\dots,Y_s$ be the irreducible components of the tangent cones at infinity $C_{\infty }(X)$ and $C_{\infty }(Y)$ respectively. If $X$ and $Y$ are outer bi-Lipschitz homeomorphic at infinity, then $r=s$ and, up to a re-ordering of indices,  $k_X^{\infty }(X_j)=k_Y^{\infty }(Y_j)$, $\forall \ j$.
\end{theorem}

\begin{proof}
This proofs shares its structure with the corresponding result in the local case in \cite{FernandesS:2016}.
By hypotheses there are compact subsets $K\subset \C^n$ and $\T K\subset \C^m$ and an outer bi-Lipschitz homeomorphism $\varphi:X\setminus K\to Y\setminus \T K$.  
Let $S=\{n_k\}_{k\in\N}$ be a sequence of positive real numbers such that 
$$n_k\to +\infty  \quad\mbox{and} \quad \frac{\varphi(n_kv)}{n_k}\to d\varphi(v),$$ 
where $d\varphi$ is a tangent map at infinity of $\varphi$ like in Theorem \ref{inv_cones} (for more details, see \cite{FernandesS:2017}, Theorem 4.5). Since, $d\varphi$ is an outer bi-Lipschitz homeomorphism, we get $r=s$ and there is a permutation $P\colon\{1,\dots,r\}\to \{1,\dots,s\}$ such that $d\varphi (X_j)=Y_{P(j)}$ $\forall \ j.$ This is why we can suppose $d\varphi(X_j)=Y_j$ $\forall \ j$ up to a re-ordering of indices.

Let $\pi\colon\C^n\to \C^p$ be a linear projection such that 
$$\pi^{-1}(0)\cap(C_{\infty }(X)\cup C_{\infty }(Y))=\{0\}.$$ 
Let us denote the ramification locus of 
$$\pi_{| X}\colon X\to \C^p \quad \mbox{and} \quad \pi_{| C_{\infty }(X)}\colon C_{\infty }(X)\to \C^p$$ 
by $\sigma(X)$ and $\sigma(C_{\infty }(X))$ respectively. By similar way, we define $\sigma(Y)$ and $\sigma(C_{\infty }(Y))$. Let us denote $\Sigma=\pi_X^{-1}(\sigma(X))$, $\Sigma'=\pi|_{C_{\infty }(X)}^{-1}(\sigma(C_{\infty }(X)))$, $\T \Sigma=\pi|_Y^{-1}(\sigma(Y))$ and $\T \Sigma'=\pi|_{C_{\infty }(Y)}^{-1}(\sigma(C_{\infty }(Y)))$.

Let us suppose that there is $j\in\{1,\dots,r\}$ such that $k_X^{\infty}(X_j)> k_Y^{\infty}(Y_j)$. Thus, given a unitary point $v\in X_j\setminus (C_{\infty }(\Sigma)\cup C_{\infty }(\Sigma'))$ such that $w=d\varphi(v)\in Y_j\setminus (C_{\infty }(\T \Sigma)\cup C_{\infty }(\T \Sigma'))$, let $\eta,R >0$ such that 
$$C_{\eta,R }(v'):=\{w\in \C^p|\, \exists t>0; \|tv'-w\|<\eta t\}\setminus B_R(0)\subset \C^p\setminus \sigma(X)\cup \sigma(C_{\infty }(X))$$
and 
$$C_{\eta,R }(w'):=\{w\in \C^p|\, \exists t>0; \|tv'-w\|<\eta t\}\setminus B_R(0)\subset \C^p\setminus \sigma(Y)\cup \sigma(C_{\infty }(Y)),$$
where $v'=\pi(v)$ and $w'=\pi(d\varphi(v))$. Therefore, there are at least two different connected components $V_{ji}$ and $V_{jl}$ of $\pi^{-1}(C_{\eta,R }(v'))\cap X$ and sequences $\{z_k\}_{k\in \N}\subset V_{ji}$ and $\{w_k\}_{k\in \N}\subset V_{jl}$ such that $t_k=\|z_k\|=\|w_k\|\in S$, $\lim \frac{1}{t_k}z_k=\lim \frac{1}{t_k}w_k=v$ and $\varphi(z_k),\varphi(w_k)\in \T V_{jm}$, where $\T V_{jm}$ is a connected component of  $\pi^{-1}(C_{\eta,R }(w'))\cap Y$.

Let us choose linear coordinates $(x,y)$ in $\C^n$ such that $\pi(x,y)=x$. 

\noindent {\bf Claim.} There exist a compact subset $K\subset \C^n$ and a constant $C>0$ such that $\|y\|\leq C\|x\|$ for all $(x,y)\in Y\setminus K$. 

If this Claim is not true, there exists a sequence $\{(x_k,y_k)\}\subset Y$ such that $\lim\limits_{k\to+\infty}\|(x_k,y_k)\|=+\infty$ and $\|y_k\|> k\|x_k\|$. Thus, up to a subsequence,  one can suppose that $\lim\limits_{k\to+\infty}\frac{y_k}{\|y_k\|}=y_0$. Since $\frac{\|x_k\|}{\|y_k\|}< \frac{1}{k}$,  $(0,y_0)\in C_{\infty }(Y)$, which is a contradiction, because $y_0\not=0$, $(0,y_0)\in \pi^{-1}(0)$ and $\pi^{-1}(0)\cap C_{\infty }(Y)=\{0\}$. Therefore, the Claim is true. In particular, $V=\T V_{jm}$ is outer bi-Lipschitz homeomorphic to $C_{\eta,R }(w')$ and since $\varphi(z_k),\varphi(w_k)\in\T V_{jm}$ $\forall$ $k\in \N$, we have 
$$\|\varphi(z_k)-\varphi(w_k)\|=o_{\infty }(t_k)$$ 
and
$$d_Y(\varphi(z_k),\varphi(w_k))\leq d_V(\varphi(z_k),\varphi(w_k))=o_{\infty }(t_k),$$
where $g(t_k)=o_{\infty }(t_k)$ means $\lim\limits _{k\to \infty }\frac{g(t_k)}{t_k}=0$.
Now, since  $X$ is outer bi-Lipschitz homeomorphic to $Y$, we have $d_X(z_k,w_k)\leq o_{\infty }(t_k)$.
On the other hand, since $z_k$ and $w_k$ lie in different connected components of $\pi^{-1}(C_{\eta,R }(v'))\cap X$, there exists a constant $C>0$ such that $d_X(z_k,w_k)\geq Ct_k$, which is a contradiction.

We have proved that $k_X^{\infty}(X_j)\leq k_Y^{\infty}(Y_j)$, $j=1,\cdots,r$. By similar arguments, using that $\varphi^{-1}$ is an outer bi-Lipschitz map, we also can prove $k_Y^{\infty}(Y_j)\leq k_X^{\infty}(X_j)$, $j=1,\cdots,r$.
\end{proof}

\subsection{Degree as an outer bi-Lipschitz invariant at infinity} \label{subsection:applications}
The first application of our main results proved in the previous section is the outer bi-Lipschitz invariance of the degree of complex algebraic curves in $\C^n$.

\begin{theorem}\label{curves}
Let $X\subset\C^n$ and $Y\subset\C^m$ be complex algebraic subsets, with $\dim X=\dim Y=1$. If $X$ and $Y$ are outer bi-Lipschitz homeomorphic at infinity, then ${\rm deg}(X)={\rm deg}(Y)$.
\end{theorem}
\begin{proof}
Let $X_1,\dots,X_r$ and $Y_1,\dots,Y_s$ be the irreducible components of the tangent cones at infinity $C_{\infty }(X)$ and $C_{\infty }(Y)$ respectively. Since $\dim X=\dim Y=1$, we have that $X_1,\dots,X_r$ and $Y_1,\dots,Y_s$ are complex lines. Thus, 
$${\rm deg}(X_1)=\dots={\rm deg}(X_r)={\rm deg}(Y_1)=\dots={\rm deg}(Y_s)=1$$
and using Eq. \ref{kurdyka-raby}, we get ${\rm deg}(X)=\sum\limits_{j=0}^r k_X^{\infty }(X_j)$ and ${\rm deg}(Y)=\sum\limits_{j=0}^s k_X^{\infty }(Y_j)$. Therefore, by Theorem \ref{multiplicities}, ${\rm deg}(X)={\rm deg}(Y)$.
\end{proof}

Let us fix $d\in \N$.

\begin{theorem}\label{main-theorem}
The statements below are equivalent.
\begin{enumerate}
\item[\~A1(d)] Let $X\subset \C^n$ and $Y\subset \C^m$ be two complex analytic sets with $\dim X=\dim Y=d$. If their germs at $0\in\C^n$ and $0\in\C^m$, respectively, are outer bi-Lipschitz homeomorphic,  then  $m(X,0)=m(Y,0)$.
\item[A1(d)] Let $X\subset \C^n$ and $Y\subset \C^m$ be two complex algebraic sets with $\dim X=\dim Y=d$. If $X$ and $Y$ are outer bi-Lipschitz homeomorphic at infinity, then  ${\rm deg}(X)={\rm deg}(Y).$
\end{enumerate}
\end{theorem}

\begin{proof} 

As we pointed out  in the introduction of the paper; we know from  \cite{FernandesS:2016} that statement \~A1($d$) holds true if and only 
if it is true by considering just homogeneous complex algebraic sets. But, if $A\subset \C^n$ is a homogeneous complex 
algebraic set, then ${\rm deg}(A)=m(A,0)$. 

From now, we are ready to start the proof of the theorem. First, let us suppose that statement A1($d$) is true. 
Since cones which are outer bi-Lipschitz homeomorphic as germs at their vertices are globally outer bi-Lipschitz homeomorphic, 
as was remarked in \cite{Sampaio:2016}, it follows from observations above that  \~A1($d$) holds true as well. 
Second, let us suppose that \~A1($d$) holds true. Let $X \subset\C^n$ and $Y\subset \C^m$ be two complex algebraic sets with $d=\dim X=\dim Y$. Let us suppose that $X$ and $Y$ are outer bi-Lipschitz homeomorphic at infinity. Then, there exist $K\subset \C^n$ and $\T K\subset\C^m$ two compact subsets and a outer bi-Lipschitz homeomorphism $\varphi\colon X\setminus K\to Y\setminus \T K$. Let us denote  by $X_1,\dots,X_r$ and $Y_1,\dots,Y_s$ the irreducible components of the cones $C_{\infty }(X)$ and $C_{\infty }(Y)$ respectively. It comes from Theorem \ref{multiplicities} that $r=s$ and the outer bi-Lipschitz homeomorphism $d\varphi \colon C_{\infty }(X)\rightarrow C_{\infty }(Y)$, up to re-ordering of indices, sends $X_i$ onto $Y_i$ and $k_X^\infty(X_i)=k_Y^\infty(Y_i)$ $\forall$ $i$. Furthermore, $d\varphi(0)=0$.

By Proposition \ref{algebricity}, the tangent cones at infinity $C_{\infty }(X)$ and $C_{\infty }(Y)$ are homogeneous complex algebraic subsets. Thus, the irreducible components $X_1,\dots,X_r$ and $Y_1,\dots,Y_s$ are homogeneous complex algebraic subsets as well. Since \~A1($d$) is true, we have $m(X_i,0)=m(Y_i,0)$ $\forall \ i$, hence ${\rm deg}(X_i)={\rm deg}(Y_i)$ $\forall \ i$. Finally, by using Eq. \ref{kurdyka-raby}, we get ${\rm deg}(X)={\rm deg}(Y)$ which give us that  A1($d$) is true.
\end{proof}

\begin{theorem}\label{main-result}
\begin{itemize}
\item [(1)] Let $X\subset \CC^{N+1}$ and $Y\subset \CC^{M+1}$ be two complex analytic surfaces. 
If $(X,0)$ and $(Y,0)$ are outer bi-Lipschitz homeomorphic, then $m(X,0)=m(Y,0)$.
\item [(2)] Let $X\subset \CC^{N+1}$ and $Y\subset \CC^{M+1}$ be two complex algebraic surfaces. 
If $X$ and $Y$ are outer bi-Lipschitz homeomorphic at infinity, then ${\rm deg}(X)={\rm deg}(Y)$.
\end{itemize}
\end{theorem}

\begin{proof}
By Theorem 2.1 in \cite{FernandesS:2016}, it is enough to show (1) when $X$ and $Y$ are two irreducible homogeneous complex algebraic sets. 
The Theorem 2.1 in \cite{FernandesS:2016} was stated with the outer bi-Lipschitz homeomorphism taking account the ambient spaces where the germs were embedded, however, let us emphasize here that the 
same proof of the Theorem 2.1 in \cite{FernandesS:2016} works when we consider the outer bi-Lipschitz homeomorphism between the germs does not taking account the ambient spaces. 
Thus, the part (1) of the theorem is reduced to prove the following proposition.

\begin{proposition}
Let $(S,0)$ and $(S',0)$ be homogeneous irreducible affine algebraic surfaces. If $S\setminus\{0\}$ is homeomorphic to $S'\setminus\{0\}$, then we have the equality ${\rm deg}(S)={\rm deg}(S')$.
\end{proposition}
\begin{proof}
We have the isomorphism $H^2(S\setminus\{0\};\Z)\cong H^2(S'\setminus\{0\};\Z)$. Hence it is enough to show that the torsion part of this cohomology group is isomorphic to $\Z/{\rm deg}(S)\Z$.

Let $\pi:S\setminus\{0\}\to\bbP(S)$ denote the quotient map by the $\CC^*$-action. The fibration $\pi$ is a pullback of the tautological bundle minus the zero section over the projective space $\bbP^N$
where $\bbP(S)$ is embedded, and hence the local systems $R^q\pi_*\underline{\Z}_{S\setminus\{0\}}$ are pullbacks from the corresponding local systems over $\bbP^N$. 
Since the projective space is simply-connected the local systems have trivial monodromy. Hence, the second page of the Leray spectral sequence for $\pi$ is
$$E_2^{p,q}=H^p(\bbP(S),\Z).$$

Since $E_2^{p,q}$ vanishes for $p\neq 0,1,2$ and $q\neq 0,1$, the differentials $d_i$ are all zero for $i\geq 3$ and the only non-zero $d_2$ differential is:
$$d_2:H^0(\bbP(S),\Z)\cong\Z\to H^2(\bbP(S),\Z)\cong\Z,$$
which coincides with multiplication by the first Chern class of the tautological bundle over $\bbP^N$. Hence $d_2$ is multiplication by ${\rm degree}(S)$. We deduce the isomorphisms
$$E_\infty^{0,2}=E_2^{0,2}=0,$$
$$E_\infty^{1,1}=E_2^{1,1}=\Z^{b_1},$$
$$E_\infty^{2,0}=E_3^{2,0}=\Z/{\rm deg}(S)\Z,$$
where $b_1$ is the first Betti number of $\bbP(S)$.

By the vanishing of $E_\infty^{0,2}$ there is a short exact sequence
$$0\to E_\infty^{2,0}\to H^2(S\setminus\{0\},\Z)\to E_\infty^{1,1}\to 0,$$
which splits by the freeness of $E_\infty^{1,1}$. Hence the torsion part of $H^2(S\setminus\{0\},\Z)$ is isomorphic to $\Z/{\rm deg}(S)\Z$ as needed.
\end{proof}

Now we prove part (2) of the main theorem. Since we have proved part (1), we see that the statement $\tilde A1$($2$) of 
Theorem \ref{main-theorem} has a positive answer, hence $A1$($2$) has a positive answer as well. Therefore, part (2) is proved.
\end{proof}


Let us denote by $\mathcal{C}_{1,\infty }$ the set of all complex algebraic sets $X\subset \C^n$, such that each irreducible component $X_j$ of $C_{\infty }(X)$ satisfies $\dim {\rm Sing}(X_j)\leq 1$.

\begin{theorem}\label{application1}
Let $f,g\colon \C^n\to \C$ be two polynomials. Suppose that $V(f)\in \mathcal{C}_{1,\infty }$. Suppose there exist compact subsets $K,\T K\subset \C^n$ and an outer bi-Lipschitz homeomorphism $\varphi:\C^n\setminus K\to \C^n\setminus \T K$ such that $\varphi(V(f)\setminus K)=V(g)\setminus \T K$. Then $V(g)\in \mathcal{C}_{1,\infty }$ and ${\rm deg}(V(f))={\rm deg}(V(g))$.
\end{theorem}

\begin{proof} 
By the proof of Theorem \ref{multiplicities}, we can suppose that $f$ and $g$ are irreducible homogeneous polynomials.
By Theorem 5.4 in \cite{Sampaio:2017} and by using observations in the very beginning of the proof of  
Theorem \ref{multiplicities}, it follows that ${\rm deg}(V(f)))={\rm deg}(V(g))$.
\end{proof}

\end{document}